\documentclass[12pt,leqno,headings]{amsart}
\usepackage[utf8]{inputenc}
\usepackage[T1]{fontenc}
\usepackage[english]{babel}
\usepackage{amsmath}
\usepackage{amsfonts}
\usepackage{amssymb}
\usepackage{amsthm}
\usepackage{amscd}
\usepackage{enumerate}
\usepackage{extpfeil}
\usepackage{fullpage}
\usepackage{mathtools}
\usepackage{mathdots}
\usepackage{mathrsfs}
\usepackage{nameref}
\usepackage[pdfauthor={Niels Lindner},pdftitle={Density of quasismooth hypersurfaces in simplicial toric varieties}]{hyperref}
\usepackage{amsrefs}

\renewcommand{\dim}{\operatorname{dim}}

\newcommand{\Hom}{\operatorname{Hom}}

\newcommand{\Spec}{\operatorname{Spec}}

\newcommand{\length}{\operatorname{length}}

\newcommand{\tensor}{\otimes}

\newcommand{\reg}{\mathrm{reg}}

\newcommand{\Cl}{\operatorname{Cl}}

\newtheorem{theorem}{Theorem}[section]

\newtheorem{corollary}[theorem]{Corollary}
\newtheorem{lemma}[theorem]{Lemma}

\theoremstyle{definition}
\newtheorem{definition}[theorem]{Definition}
\newtheorem*{question}{Question}
\newtheorem{example}[theorem]{Example}

\theoremstyle{remark}
\newtheorem{remark}[theorem]{Remark}
\newtheorem{step}[]{Step}

\begin{document}

\title[Density of quasismooth hypersurfaces]{Density of quasismooth hypersurfaces\\in simplicial toric varieties}
\author{Niels Lindner}
\date{\today}

\begin{abstract}
This paper investigates the density of hypersurfaces in a projective normal simplicial toric variety over a finite field having a quasismooth intersection with a given quasismooth subscheme. The result generalizes the formula found by B. Poonen for smooth projective varieties. As an application, we further analyze the density of hypersurfaces with bounds on their number of singularities and on the length of their singular schemes.
\end{abstract}

\maketitle

\section{Introduction}

Let $Y \subseteq \mathbb P^n$ be a smooth projective variety over a field $K$. Fix an integer $k \geq 1$.

\begin{question}
If $f \in H^0(\mathbb P^n, \mathcal O_{\mathbb P^n}(k))$ is a homogeneous polynomial of degree $k$ chosen uniformly at random, what is the probability that the intersection $Y\cap V(f)$ is smooth?
\end{question}

For algebraically closed fields $K$, the classical theorem of Bertini implies that the locus of hypersurfaces $f$ such that $Y \cap V(f)$ is smooth is open and dense, see e.g. \cite{Hartshorne}*{Theorem~II.8.18}. If $K = \mathbb F_q$ is a finite field with $q$ elements, then Poonen \cite{Poonen}*{Theorem~1.1} showed that
\begin{align*}
\lim_{k \to \infty} \frac{\#\{f \in H^0(\mathbb P^n, \mathcal O_{\mathbb P^n}(k)) \mid Y \cap V(f) \text{ is smooth}\}}{\#H^0(\mathbb P^n, \mathcal O_{\mathbb P^n}(k))} = \frac{1}{\zeta_Y(\dim Y + 1)},
\end{align*}
where $\zeta_Y$ is the Hasse-Weil zeta function of the projective variety $Y$, given by
$$ \zeta_Y(s) := \prod_{P \in Y \text{ closed point}} \left(1-q^{-s \deg P}\right)^{-1} = \exp \sum_{r = 1}^\infty \#Y(\mathbb F_{q^r})\frac{q^{-rs}}{r}$$
for $s \in \mathbb C$, $\mathrm{Re}(s) > \dim Y$.

The aim of this paper is to investigate how Poonen's result extends to quasismooth subschemes $Y$ of simplicial toric varieties, for example weighted projective spaces. Since we cannot expect smoothness anymore, we require instead that $Y$ be quasismooth (see Definition~\ref{D:qsm}) and ask for quasismooth intersections. The main result is the following:

\begin{theorem}\label{T:main}
Let $X$ be a projective normal simplicial toric variety over a finite field $\mathbb F_q$. Fix a Weil divisor $D$ and an ample Cartier divisor $E$ on $X$. Let $Y \subseteq X$ be any quasismooth subscheme such that $Y$ meets the singular locus of $X$ only in finitely many points. Then
\begin{align*}
\lim_{k \to \infty} &\frac{\#\{f \in H^0(X, \mathcal O_{X}(D+kE)) \mid Y \cap V(f) \text{ is quasismooth}\}}{\#H^0(X, \mathcal O_{X}(D+kE))} = \prod_{P \in Y \text{ closed}} \left(1-q^{-\nu_P(D)}\right),
\end{align*}
where $\nu_P(D)$ is a non-negative integer depending on $P$ and $D$ with the property that $\nu_P(D)$ equals $ \deg P \cdot (\dim Y + 1)$ if $Y$ is smooth at $P$.
\end{theorem}

\begin{remark}
\begin{enumerate}[(1)]
 \item If $X = \mathbb P^n$, $D = 0$ and $E$ is a hyperplane, then we recover Poonen's result  \cite{Poonen}*{Theorem~1.1}.
 \item In \cite{Semiample}, Poonen's formula is generalized to a semiample setting. In the special case that $X$ is a smooth toric variety, the result \cite{Semiample}*{Theorem~1.1} implies our Theorem~\ref{T:main}.
 \item For a precise definition of the number $\nu_P(D)$ and its properties we refer to Subsection~\ref{SS:nuP}.
\item The formula in Theorem~\ref{T:main} is in particular valid if $\nu_P(D) = 0$ for some closed point $P \in Y$. In this case, both sides of the equation are zero. Moreover, $Y \cap V(f)$ fails to be quasismooth for all $f \in H^0(X, \mathcal O_X(D+kE))$ and all $k \geq 0$, see Corollary~\ref{C:low}. For a situation where $\nu_P(D) = 0$ occurs, see Example~\ref{E:wps1}. However, if $X$ is smooth or $D$ is Cartier, then $\nu_P(D)$ is always positive by Lemma~\ref{L:nuP}.
\item If the intersection of $Y$ with the singular locus of $X$ is of positive dimension, then Theorem~\ref{T:main} may fail, see Lemma~\ref{L:medium} and Example~\ref{E:wps2}.
\end{enumerate}
\end{remark}

The proof uses a modified version of Poonen's closed point sieve: We divide the closed points of $Y$ into low, medium and high degree points and show that the impact of the latter two is negligible. 

At first, we need to develop some preliminaries on simplicial toric varieties and quasismoothness. After an extensive study of restriction maps to zero-dimensional subschemes and the numbers $\nu_P(D)$ in Section~\ref{S:sections}, we can finally adapt Poonen's strategy to prove Theorem~\ref{T:main} in Section~\ref{S:sieve}.

In Section~\ref{S:applications} we give a formula for the density of quasismooth hypersurfaces with an upper bound on the number of singular points and the length of the singular schemes, respectively. Finally, for smooth toric varieties, we show that hypersurfaces of degree $k$ whose singular scheme is of length at least $k$ form a set of density zero.

\section{Facts on simplicial toric varieties}

We first collect some facts on toric varieties. Let $X$ be an $n$-dimensional projective normal simplicial split toric variety without torus factors over a perfect field $K$. Let $\Sigma$ be the corresponding simplicial fan in the lattice $N \cong \mathbb Z^n$. Denote by $M$ the dual lattice of $N$ and set $d$ to be the number of one-dimensional cones in $\Sigma$.

\subsection{The homogeneous coordinate ring (\cite{Cox}, \cite{CLS}*{§5.2, §5.3})}  Denote by $\Cl(X)$ the class group of $X$, i. e. the group of Weil divisors on $X$ modulo rational equivalence. Then there is an exact sequence
$$ 0 \to M \to \mathbb Z^{d} \xrightarrow{\vartheta} \Cl(X) \to 0 $$
of abelian groups.

The homomorphism $\vartheta$ induces a grading by the class group on the polynomial ring $S := K[x_1, \dots, x_d]$: If $x_1^{\alpha_1} \cdots x_d^{\alpha_d} \in S$ is a monomial, define $\deg(x_1^{\alpha_1} \cdots x_d^{\alpha_d}) := \vartheta(\alpha_1,\dots,\alpha_d)$. $S$ is called the homogeneous coordinate ring of $X$. Let $D, E$ be Weil divisors on $X$. There is a natural isomorphism
$$S \cong \bigoplus_{[D] \in \Cl(X)} H^0(X, \mathcal O_X(D))$$
of graded rings, which is compatible with the natural multiplication maps of sections
$$H^0(X, \mathcal O_X(D)) \otimes H^0(X, \mathcal O_X(E)) \to H^0(X, \mathcal O_X(D+E)).$$
Any finitely generated graded $S$-module $M$ gives rise to a coherent sheaf $\widetilde M$ on $X$, and conversely every coherent sheaf on $X$ arises this way. Furthermore, every homogenenous ideal of $S$ defines a closed subscheme of $X$, and every closed subscheme of $X$ comes from some homogeneous ideal of $S$.

\subsection{Quotient construction (\cite{CLS}*{§5.1})} Define $G := \Hom_{\mathbb Z}(\Cl(X),\mathbb G_m)$. If $L/K$ is a field extension, then $G$ acts on the $L$-rational points of $\mathbb A^d = \Spec S = \Spec K[x_1,\dots,x_d]$ via
$$ G \times L^d \to L^d, \quad g \cdot (a_1,\dots,a_d) \mapsto (g(\deg(x_1)) \cdot a_1, \dots, g(\deg(x_d)) \cdot a_d). $$
There is an algebraic set $B \subseteq \mathbb A^d$ depending on $X$, such that the quotient $\mathbb A^d \setminus B \to (\mathbb A^d \setminus B)/G$ is geometric and isomorphic to $X$. The singularities of $X$ are all due to the quotient action of $G$, which acts with finite isotropy groups on $\mathbb A^d \setminus B$.

\subsection{Quasismoothness}\label{SS:qsm}

\begin{definition}[see \cite{CoxQsm}*{Definition~3.1}]\label{D:qsm}
Denote by
$$\pi: \mathbb A^d \setminus B \to (\mathbb A^d \setminus B)/G \cong X$$
the quotient map. 
\begin{itemize}
  \item A subscheme $Y \subseteq X$ is called \textit{quasismooth at a closed point} $P \in Y$ if $\pi^{-1}(Y)$ is smooth at all points in the fiber $\pi^{-1}(P)$.
   \item $Y$ is called \textit{quasismooth} if it is quasismooth at all closed points.
\end{itemize}
\end{definition}

\begin{remark}\label{R:qsm} As above, let $X$ be any projective normal simplicial toric variety, and let $Y \subseteq X$ be a subscheme.
\begin{enumerate}
  \item $X$ is quasismooth.
  \item If $Y$ is smooth at $P \in Y$, it is also quasismooth at $P$.
  \item If $Y$ is quasismooth at $P \in Y$ and $X$ is smooth at $P$, then $Y$ is smooth at $P$.
  \item\label{I:qsm4} $Y$ is quasismooth at a closed point $P$ if and only if $\pi^{-1}(Y)$ is smooth at some point in $\pi^{-1}(P)$.
\end{enumerate}
\end{remark}

Testing quasismoothness means testing smoothness on the affine cone: For example, if $X = \mathbb P^n$, then a subscheme $Y \subseteq X$ is quasismooth if and only if the affine cone of $Y$ is smooth outside $\{0\}$. This is in turn equivalent to $Y$ being smooth.

Moreover, if $Y$ is a (quasi)smooth subscheme of $\mathbb P^n$ and $f$ is a homogeneous polynomial in $n+1$ variables, then the affine cone of $Y \cap V(f)$ is not smooth at a point $Q \in \mathbb A^{n+1}\setminus\{0\}$ if and only if the order of vanishing of $f$ at $Q$ is at least two. Such polynomials form a homogeneous ideal inside the polynomial ring in $n+1$ variables. This defines in turn a closed subscheme of $\mathbb P^n$. 

This works in general: Let $X$ be an arbitrary projective normal simplicial toric variety and let $Y$ be a quasismooth subscheme of $X$. Pick a global section $f \in H^0(X, \mathcal O_X(D))$ of some Weil divisor $D$ on $X$. Then the quasismoothness of $Y \cap V(f)$ is still a local condition on $Y$: If $P$ is a closed point of $Y$, we pull back the first-order infinitesimal neighborhood of all points in the affine quasicone lying over $P$. More precisely, we have the following:

\begin{lemma}\label{L:qsm}  Let $Y \subseteq X$ be a quasismooth subscheme, $P \in Y$ a closed point.
Then there is a closed subscheme $Y_P \subseteq X$ such that for all Weil divisors $D$ on $X$ and $f \in H^0(X, \mathcal O_X(D))$ we have
$$ Y \cap V(f) \text{ is quasismooth at } P \quad \Leftrightarrow \quad \varphi_{P,D}(f) \neq 0,$$
where $\varphi_{P,D}$ is the natural restriction map
$$ \varphi_{P,D}: \quad H^0(X,\mathcal O_X(D)) \to H^0(Y_P,\mathcal O_X(D)|_{Y_P}).$$
\end{lemma}

\begin{proof}
Let $S$ be the homogeneous coordinate ring of $X$ and $\pi: \mathbb A^d \setminus B \to X$ the map from the quotient construction. For any $Q \in \pi^{-1}(P)$, there are natural maps
$$ \vartheta_Q:\quad S \to \mathcal O_{\pi^{-1}(Y),Q} \to \mathcal O_{\pi^{-1}(Y),Q}/\mathfrak m_Q^2, $$
where $\mathfrak m_Q$ is the maximal ideal of the local ring $\mathcal O_{\pi^{-1}(Y),Q}$ of $\pi^{-1}(Y)$ at $Q$. Denote by $I_P$ the largest homogeneous ideal of $S$ contained in $\bigcap_{Q \in \pi^{-1}(P)} \ker \vartheta_Q$ with respect to the grading given by $\Cl(X)$. Then $I_P$ defines a closed subscheme $Y_P$ of $X$.

Let $D$ be a Weil divisor on $X$. For $f \in S_{[D]}$, the intersection $Y \cap V(f)$ is not quasismooth at $P$ if and only if there is a point $Q \in \pi^{-1}(P)$ such that $\vartheta_Q(f) = 0$. By Remark~\ref{R:qsm}~(\ref{I:qsm4}), this is equivalent to $\vartheta_Q(f) = 0$ for all $Q \in \pi^{-1}(P)$, which is in turn equivalent to $f \in I_P \cap S_{[D]}$. In other words, $f$ lies in $\ker \varphi_{P,D}$, after applying the isomorphism $S_{[D]} \cong H^0(X, \mathcal O_X(D))$.
\end{proof}

\begin{example}\label{E:qsm}
 Assume $Y = X$ and let $P \in X$ be a closed point. By definition, the ideal $I_P$ inside the homogeneous coordinate ring $S = K[x_1,\dots,x_d]$ is generated by all homogeneous polynomials $f \in S$ such that
$$f(Q) = \frac{\partial f}{\partial x_1}(Q) = \dots = \frac{\partial f}{\partial x_d}(Q) = 0$$
for all $Q \in \pi^{-1}(P)$. Quasismoothness can hence be effectively tested with the Jacobian criterion on $\mathbb A^d \setminus B$.

Moreover, if $\mathfrak p$ is the prime ideal of $S$ corresponding to the point $P$, then $S/\mathfrak p$ is an integral domain. In particular, the fiber $\pi^{-1}(P)$ is an integral scheme over a perfect field and hence generically smooth. Since $S$ is a regular ring, we can invoke \cite{EH}*{Corollary~1} to obtain $$I_P = \mathfrak p^{(2)},$$
where $\mathfrak p^{(2)}$ denotes the symbolic square of $\mathfrak p$.

 More generally, let $Y \subseteq X$ be a closed quasismooth subscheme cut out by a homogeneous ideal $J_Y$ with respect to the grading by the class group $\Cl(X)$. Let $P \in Y$ be a closed point and denote by $\mathfrak p$ the prime ideal of $S$ defining $P$ in $X$. Then $\pi^{-1}(P)$ is generically smooth as above. Furthermore, since $Y$ is quasismooth, $\pi^{-1}(Y)$ is smooth and its coordinate ring is hence regular. We can apply \cite{EH}*{Corollary~1} again to see
$$ I_P = J_Y + \mathfrak p^{(2)}.$$
\end{example}
  
\section{Sections restricted to zero-dimensional subschemes}\label{S:sections}

Let $X$ be as above, $Y \subseteq X$ a quasismooth subscheme. Fix a Weil divisor $D$ and an ample Cartier divisor $E$ on $X$. We want to determine the proportion of sections of $D+kE$ having a quasismooth intersection with $Y$ as $k \to \infty$.

In view of Lemma~\ref{L:qsm}, we will take a closer look at the $K$-vector spaces $H^0(Y_P, \mathcal O_X(D)|_{Y_P})$ and the map $\varphi_{P,D}$.

\subsection{Surjectivity of $\varphi_{P,D}$}\label{SS:surj}
Let $Z$ be a zero-dimensional subscheme of $X$ and denote the corresponding closed immersion by $i: Z \hookrightarrow X$. Then there is an associated surjective map $\mathcal O_X \twoheadrightarrow i_*\mathcal O_Z$ of sheaves. Tensoring with $\mathcal O_X(D)$, taking the long exact sequence in cohomology and applying the projection formula, this yields a natural map on global sections
$$ \varphi_Z: H^0(X, \mathcal O_X(D)) \to H^0(Z, \mathcal O_X(D)|_Z).$$
This way, we recover $\varphi_{P,D}$ if $Z$ equals the scheme $Y_P$. Tensoring with $\mathcal O_X(D+kE)$ instead of $\mathcal O_X(D)$, we obtain 
$$ \varphi_{Z,k}: H^0(X, \mathcal O_X(D+kE)) \to H^0(Z, \mathcal O_X(D+kE)|_Z) \cong H^0(Z, \mathcal O_X(D)|_Z).$$
The last isomorphism comes from the fact that 
$$\mathcal O_X(D+kE) \cong \mathcal O_X(D) \otimes \mathcal O_X(E)^{\otimes k},$$ since $X$ is normal, and that $\mathcal O_X(E)$ is locally free of rank one, as $E$ is Cartier.

We see that $\varphi_{Z,k}$ is surjective if $H^1(X,\mathcal K \otimes \mathcal O_X(kE))$ vanishes, where $\mathcal K$ is the kernel of the surjection $\mathcal O_X(D) \to  \mathcal O_X(D)|_Z$. Since $\mathcal K$ is a coherent sheaf on the projective variety $X$ and $E$ is ample, we have the following result by Serre vanishing \cite{Hartshorne}*{Theorem~II.5.3}:

\begin{lemma}\label{L:surj1}
For any zero-dimensional subscheme $Z \subseteq X$ exists an integer $k_Z$ such that the natural map
$$\varphi_{Z,k}: H^0(X, \mathcal O_X(D+kE)) \to H^0(Z, \mathcal O_X(D)|_Z)$$
is surjective for all $k \geq k_Z$.
\end{lemma}

The rest of this subsection is devoted to an improvement of this result. In order to achieve this, we need to have a look at multiplication of sections on toric varieties. Define $\reg_E(D)$ to be the smallest integer $\ell \geq 1$ such that
$$H^i(X, \mathcal O_X(D+kE - iE)) = 0 \quad \text{ for all } k \geq \ell \text{ and } i \geq 1.$$
The number $\reg_E(D)$ exists and coincides with the Castelnuovo-Mumford regularity of the sheaf $\mathcal O_X(D)$ with respect to the ample line bundle $\mathcal O_X(E)$ on $X$.

\begin{lemma}\label{L:multiplication}
The natural multiplication map
$$ H^0(X, \mathcal O_X(D+kE)) \otimes H^0(X, \mathcal O_X(E)) \to H^0(X, \mathcal O_X(D+(k+1)E))$$
is surjective for all $k \geq \reg_E(D)$.
\end{lemma}
\begin{proof}
See \cite{Mumford}*{Theorem~2}.
\end{proof}

We will now give an enhanced version of Lemma~\ref{L:surj1}:

\begin{lemma}\label{L:surj2}
For all zero-dimensional subschemes $Z$ the map $\varphi_{Z,k}$ is surjective whenever
$$k \geq \dim_{K} H^0(Z, \mathcal O_X(D)|_Z) + \reg_E(D) - 1.$$
\end{lemma}
\begin{proof}
Let $Z$ be a zero-dimensional subscheme of $X$. Since cohomology commutes with flat base change, we can check the surjectivity of the map $\varphi_{Z,k}$ after a base change to some field extension. Thus we can w.l.o.g. assume the existence of a section $f_0 \in H^0(X ,\mathcal O_X(E)) \cong S_{[E]}$ defined over $K$ satisfying $V(f_0) \cap Z = \emptyset$. Choose elements $f_1, \dots, f_s \in S_{[E]}$ such that $\{f_0, \dots, f_s\}$ forms a $K$-basis of $S_{[E]}$.

By Lemma~\ref{L:multiplication}, we have surjective multiplication maps
$$ H^0(X, \mathcal O_X(D+\ell E)) \tensor H^0(X,\mathcal O_X(E))^{\tensor k-\ell} \twoheadrightarrow H^0(X,\mathcal O_X(D+kE)), \quad k \geq \ell := \reg_E(D).$$
These maps are compatible with the isomorphisms $H^0(X,\mathcal O_X(-)) \cong S_{[-]}$. Identify now  $H^0(X,\mathcal O_X(E))^{\otimes (k-\ell)}$ with the space of homogeneous polynomials in $f_0, \dots, f_s$ of degree $k-\ell$, where $\ell := \reg_E(D)$. Homogenization via $f_0$ yields an isomorphism
\begin{align*}
S_{[D+\ell E]}\otimes K[f_1,\dots,f_s]_{\leq k-\ell} \cong S_{[D+\ell E]} \otimes K[f_0,\dots,f_s]_{k-\ell}
\end{align*}
and we thus obtain a surjective $K$-linear map
\begin{align*}
S_{[D+\ell E]}\, \otimes\, K[f_1,\dots,f_s]_{\leq k-\ell}
 &\twoheadrightarrow S_{[D+\ell E]} \otimes S_{[E]}^{\otimes(k-\ell)}
\\ &\cong H^0(X,\mathcal O_X(D+\ell E)) \otimes H^0(X,\mathcal O_X(E))^{\otimes (k-\ell)}
\\ &\twoheadrightarrow H^0(X,\mathcal O_X(D+kE)).
\end{align*}
Consider the composition
\begin{align*}
\vartheta_k: S_{[D+\ell E]} \otimes K[f_1,\dots,f_s]_{\leq k-\ell} \to H^0(X, \mathcal O_X(D+kE))
\xrightarrow{\varphi_{Z,k}} H^0(Z, \mathcal O_X(D)|_Z).
\end{align*}
The linear map $\vartheta_k$ becomes surjective for large enough $k$ by Lemma~\ref{L:surj1}. Furthermore, if $\vartheta_k$ is surjective, then so is $\varphi_{Z,k}$. Define the subspaces
$$ B_j := \vartheta_k\left(S_{[D+\ell E]} \otimes K[f_1,\dots,f_s]_{\leq j} \right),\quad j = -1, \dots, k-\ell.$$
This yields an ascending chain of subspaces $\{0\} = B_{-1} \subseteq B_{0} \subseteq ...$, thus for some $j \geq -1$ holds $B_j = B_{j+1}$. Then, if $[f_i]$ denotes the image of $f_i$ in $H^0(Z, \mathcal O_Z)$, we obtain
\begin{align*}
B_{j+2} = \sum_{i=1}^s [f_i] \cdot B_{j+1} =  \sum_{i=1}^s [f_i] \cdot B_{j} = B_{j+1}.
\end{align*}
A fortiori, $B_r = B_j$ for $r \geq j$. But $\vartheta_k$ is eventually surjective, so as soon as $B_j = B_{j+1}$, it must be the all of $H^0(Z, \mathcal O_X(D)|_Z)$ for large $k$. This means that $\vartheta_k$ and hence $\varphi_{Z,k}$ are surjective whenever
$$k - \ell \geq \dim_K H^0(Z, \mathcal O_X(D)|_Z) - 1.$$
\end{proof}

\subsection{Dimension of $H^0(Y_P, \mathcal O_X(D)|_{Y_P})$}\label{SS:nuP} 
\begin{definition}\label{D:nuP}
With the same notation as above, define
$$ \nu_P(D) := \dim_K H^0(Y_P, \mathcal O_X(D)|_{Y_P}).$$ 
\end{definition}

\begin{remark}
A general recipe to compute $\nu_P(D)$ is the following: Let $\pi: \mathbb A^d\setminus B \to X$ denote the quotient map. Pick a closed point $P \in Y$. By Lemma~\ref{L:qsm}, a section $f \in H^0(X, \mathcal O_X(D+kE))$ lies in the kernel of
$$\varphi_{Y_P,k}: H^0(X, \mathcal O_X(D+kE)) \to H^0(Y_P, \mathcal O_{X}(D)|_{Y_P})$$
if and only if $V(f)$ is not quasismooth at $P$, i. e. if and only if $\pi^{-1}(V(f))$ is not smooth at some point $Q \in \pi^{-1}(P)$. The latter condition can be tested with the Jacobian criterion and gives therefore an effective description of $\ker \varphi_{Y_P,k}$. Since $\varphi_{Y_P,k}$ is surjective for $k \gg 0$ by Lemma~\ref{L:surj1}, this computes the number $\nu_P(D)$ as the codimension of $\ker \varphi_{Y_P,k}$ in $H^0(X, \mathcal O_X(D+kE))$.
\end{remark}

\begin{remark}
An alternative description is the following: If $Y \subseteq X$ is a closed subscheme, we can use the formula from Example~\ref{E:qsm}: Let $S$ denote the homogeneous coordinate ring of $X$ and let $J_Y$ be the ideal of $Y$ inside $S$. Pick a closed point $P \in Y$ and let $\mathfrak p$ denote the corresponding prime ideal in $S$. Then, for $k \gg 0$, $\nu_P(D)$ equals the dimension of the degree $[D+kE]$-part of the $\Cl(X)$-graded $S$-module $S/(J_Y + \mathfrak p^{(2)})$ .
\end{remark}

% \begin{example}
%   For instance, let $X = Y = \mathbb P^n$, $D = 0$, $E$ a hyperplane. Global sections of $\mathcal O_X(D+kE)$ correspond to homogeneous polynomials in $n+1$ variables of degree $k$. For such a polynomial $f$ of degree $k$ coprime to the characteristic of the ground field, the hypersurface $V(f) \subseteq \mathbb A^{n+1} \setminus \{0\}$ is not smooth at a point $Q$ if and only if $f$ and its $n+1$ partial derivatives vanish at $Q$. Since $f$ is homogeneous, these are $n+1$ linear conditions over the residue field of $P$. Hence $\nu_P(D) = \deg P \cdot (n+1)$.
% 
% Alternatively, if $\mathfrak p$ is a prime ideal corresponding to a closed point in $\mathbb P^n$, then $\mathfrak p$ can be generated by a regular sequence.
% \end{example}

\begin{example}\label{E:wps0}
 Let $X$ be the weighted projective space $\mathbb P(1,\dots,1,2)$ with the coordinates $x_0, \dots, x_n$. Furthermore, let $Y = X$, $D = V(x_n)$ and $E = V(x_0)$. Then $H^0(X, \mathcal O_X(D+kE))$ corresponds to the space of weighted homogeneous polynomials in the variables $x_0,\dots,x_n$ of degree $2k+1$. Such a polynomial $f$ can be written as
$$ f= \sum_{i=0}^k x_n^i \cdot f_i(x_0,\dots,x_{n-1}), \quad f_i \text{ homogeneous of degree } 2(k-i)+1.$$
 If $Q \in \mathbb A^{n+1} \setminus \{0\}$ lies over the singular point $P = (0:\dots:0:1)$, then one computes that both $f$ and $\frac{\partial f}{\partial x_n}$ always vanish at $Q$. Moreover, the partial derivatives $\frac{\partial f}{\partial x_0}, \dots, \frac{\partial f}{\partial x_{n-1}}$ vanish simultaneously at $Q$ if and only if $f_k = 0$. Thus $f$ lies in $\ker \varphi_{Y_P,k}$ if and only if $f_k = 0$. Since $f_k$ is a linear homogeneous polynomial in $n$ variables, this is a codimension $n$ condition. Hence $\nu_P(D) = \deg P \cdot n = n$.

Alternatively, let $\mathfrak p = \langle x_0, \dots, x_{n-1} \rangle$ be the prime ideal of the polynomial ring $S = K[x_0,\dots,x_n]$ corresponding to $P = (0: \dots :0:1)$. One checks that $\mathfrak p^{(2)} = \mathfrak p^2$, so
$$\nu_P(D) = \lim_{k \to \infty} \dim_K (S/\mathfrak p^2)_{2k+1} = n,$$
as $(S/\mathfrak p^2)_{2k+1}$ is spanned by the classes of $x_0 x_n^k, x_1 x_n^k, \dots, x_{n-1} x_n^k$.
\end{example}

For another computation of $\nu_P(D)$, see Example~\ref{E:wps1}. We summarize some properties of the number $\nu_P(D)$ in the following lemma:

\begin{lemma}\label{L:nuP} Let $P$ be a closed point of $Y$.
\begin{enumerate}[\normalfont (1)]
  \item $\nu_P(D)$ is divisible by $\deg P$.
  \item If $D$ is Cartier, then $\nu_P(D) \geq \deg P$.
  \item If $X$ is smooth at $P$, then $\nu_P(D) = \deg P \cdot (\dim Y + 1)$.
  \item In general, $\nu_P(D) \leq \deg P \cdot \dim \pi^{-1}(Y)$,
  where $\pi$ is the map from the quotient construction.
\end{enumerate}
\end{lemma}
\begin{proof}
 Recall that in the proof of Lemma~\ref{L:qsm}, $Y_P$ was defined by the homogeneous ideal $I_P$, which was the largest homogeneous ideal contained in $\bigcap_{Q \in \pi^{-1}(P)} \ker(S \to \mathcal O_{\pi^{-1}(Y),Q}/\mathfrak m_Q^2)$.
\begin{enumerate}[(1)]
\item Let $\kappa(P)$ be the residue field of $P$. Since $K$ is perfect, the field extension $\kappa(P)/K$ is separable. Suppose that $P_1, \dots, P_{\deg P}$ are the $\deg P$ distinct points lying over $P$. Denote by $X'$, $Y'$ and $D'$ the respective base changes of $X$, $Y$ and $D$ to $\kappa(P)$. Then 
$$H^0(Y_P,\mathcal O_X(D)|_{Y_P}) \otimes_ K \kappa(P) \cong \bigoplus_{i=1}^{\deg P} H^0(Y'_{P_i}, \mathcal O_{X'}(D')|_{Y'_{P_i}}),$$
where all the direct summands on the right-hand side have the same dimension over $\kappa(P)$.
\item If $D$ is Cartier, then $\mathcal O_X(D)$ is locally free and hence
$$ H^0(Y_P, \mathcal O_X(D)|_{Y_P}) \cong H^0(Y_P, \mathcal O_{Y_P}).$$
Since the latter space is of positive dimension, (1) yields the estimate $\nu_P(D) \geq \deg P$. 
 \item  Let $\mathcal O_{Y,P}$ be the local ring of $Y$ at $P$ with maximal ideal $\mathfrak m_P$. Since $\mathcal O_X(D)$ is invertible when restricted to the smooth locus, we get a honest restriction map $\rho: S \to \mathcal O_{Y,P}$. Now
\begin{align*}
f \in I_P &\Leftrightarrow Y \cap V(f) \text{ quasismooth at } P  \\
&\Leftrightarrow Y \cap V(f) \text{ smooth at } P \\
&\Leftrightarrow \rho(f) \in \mathfrak m_P^2.
\end{align*}
Since $Y$ is smooth at $P$, the $\mathbb F_q$-dimension of $\mathcal O_{Y,P}/\mathfrak m_P^2$ equals $\deg P \cdot (\dim Y + 1)$.
\item Pick a point $Q \in \pi^{-1}(P)$ of the same degree as $P$. As the restriction map $\varphi_{Y_P,k}$ is eventually surjective for large enough $k$ by Lemma~\ref{L:surj1}, $H^0(Y_P, \mathcal O_X(D)|_{Y_P})$ has the same dimension as $(S/I_P)_{[D+kE]}$ for all $k \gg 0$. But the latter space injects into $\mathcal O_{\pi^{-1}(Y),Q}/\mathfrak m_Q^2$, which has dimension $\deg Q\cdot(\dim \pi^{-1}(Y) + 1)$, as $Y$ is smooth at $Q$. Since this injection cannot be surjective, $$\nu_P(D) < \deg Q \cdot (\dim \pi^{-1}(Y) + 1) = \deg P \cdot (\dim \pi^{-1}(Y) + 1).$$
By part (1), this implies $\nu_P(D) \leq \deg P \cdot \dim \pi^{-1}(Y)$. \qedhere
\end{enumerate}
\end{proof}

\begin{corollary}\label{C:surj}
Suppose that $P$ is a closed point of $Y$ and $k$ is a positive integer such that
$$ \deg P \leq \frac{k - \reg_E(D) + 1}{\dim \pi^{-1}(Y)}.$$
Then the map $\varphi_{Y_P,k}: H^0(X, \mathcal O_X(D+kE)) \to H^0(Y_P, \mathcal O_X(D)|_{Y_P})$ is surjective.
\end{corollary}
\begin{proof}
Lemma~\ref{L:nuP} gives the bound
$$\nu_P(D) \leq \deg P \cdot \dim \pi^{-1}(Y) \leq k - \reg_E(D) + 1.$$
Thus $\varphi_{Y_P,k}$ is surjective, as $k \geq \nu_P(D) + \reg_E(D) - 1$ due to Lemma~\ref{L:surj2}.
\end{proof}

\section{Sieving closed points}\label{S:sieve}

We are now in the shape to prove Theorem~\ref{T:main} following the method of Poonen. The notations are the same as in the previous section, except that we assume $K = \mathbb F_q$ to be a finite field.

\subsection{Low degree points}

\begin{lemma}[Low degree points]\label{L:low}
For $r \geq 1$, let $Y_{<r}$ be the set of closed points of $Y$ of degree less than $r$. Then there is a positive integer $k_r$ such that for all $k \geq k_r$ holds
\begin{align*}
  &\frac{\#\{ f \in H^0(X, \mathcal O_X(D+kE)) \mid Y \cap V(f) \text{ is quasismooth at all } P \in Y_{<r} \}}{\#H^0(X, \mathcal O_X(D+kE))} \\
   &\quad= \prod_{P \in Y_{<r}} \left( 1 - q^{-\nu_P(D)}\right).
\end{align*}
\end{lemma}
\begin{proof}
Let $Z$ be the union of all schemes $Y_P$ for $P \in Y_{<r}$. By Lemma~\ref{L:qsm}, a section $f \in H^0(X, \mathcal O_X(D+kE))$ is quasismooth at all $P \in Y_{<r}$ if and only if $\varphi_{Z,k}(f)$ vanishes nowhere, where $\varphi_{Z,k}$ denotes the composition
$$H^0(X, \mathcal O_X(D+kE)) \to H^0(Z, \mathcal O_X(D)|_{Z}) \cong \prod_{P \in Y_{<r}} H^0(Y_P, \mathcal O_X(D)|_{Y_P}).$$
According to Lemma~\ref{L:surj1}, there is a constant $k_r$ such that for all $k \geq k_r$, the map $\varphi_{Z,k}$ is surjective. The fibers of a surjective linear map between finite vector spaces have all the same cardinality, hence
\begin{align*}
  &\frac{\#\{ f \in H^0(X, \mathcal O_X(D+kE)) \mid Y \cap V(f) \text{ is quasismooth at all } P \in Y_{<r} \}}{\#H^0(X, \mathcal O_X(D+kE))} \\
   &\quad= \frac{\#\varphi_{k,Y}^{-1}\left(\prod_{P \in Y_{<r}} \left(H^0(Y_P, \mathcal O_X(D)|_{Y_P}) \setminus \{0\}\right) \right)}{\#\varphi_{k,Y}^{-1}\left(\prod_{P \in Y_{<r}} H^0(Y_P, \mathcal O_X(D)|_{Y_P})  \right)}\\
      &\quad= \prod_{P \in Y_{<r}} \left( 1 - \frac{1}{\# H^0(Y_P, \mathcal O_X(D)|_{Y_P})}\right) \\
   &\quad= \prod_{P \in Y_{<r}} \left( 1 - q^{-\nu_P(D)}\right).
\end{align*}
\end{proof}

\begin{corollary}\label{C:low}
 If $\nu_P(D) = 0$ for some closed point $P$ of $Y$, then $Y \cap V(f)$ is not quasismooth at $P$ for all $f \in H^0(X, \mathcal O_X(D+kE))$ and all $k \geq 0$.
\end{corollary}
\begin{proof}
 Let $P \in Y$ be a closed point with $\nu_P(D) = 0$. In particular, $H^0(Y_P, \mathcal O_X(D)|_{Y_P}) = 0$. Then the map $\varphi_{Y_P,k}$ is surjective for all $k \geq 0$ for trivial reasons. Repeating the computation in the proof of Lemma~\ref{L:low} above shows that 
  $$\frac{\#\{ f \in H^0(X, \mathcal O_X(D+kE)) \mid Y \cap V(f) \text{ is quasismooth at } P \}}{\#H^0(X, \mathcal O_X(D+kE))} = 0.$$
\end{proof}

\begin{example}\label{E:wps1}
It can happen that $\nu_P(D) = 0$. In this case, Corollary~\ref{C:low} states that no section in $H^0(X, \mathcal O_X(D+kE))$ has quasismooth intersection with $Y$.

For example, consider the $n$-di weighted projective space
$$X = Y = \mathbb P(1,\dots,1,w)$$
of dimension $n$, where $w \geq 3$. $X$ has homogeneous coordinate ring $S = \mathbb F_q[x_0,\dots,x_n]$, and the grading by the class group $\Cl(X) \cong \mathbb Z$ is given by $\deg(x_i)=1$ for $i = 0, \dots, n-1$ and $\deg(x_n) = w$. Choose a Weil divisor $D_\ell$ corresponding to $\mathcal O_X(\ell)$, where $\ell \in \{0, \dots, w - 1 \}$. This is not Cartier if $\ell \neq 0$. However, the sheaf $\mathcal O_X(w)$ is ample and invertible. The only singular point of $X$ is $P = (0:\dots:0:1)$ in weighted homogeneous coordinates. All other points $Q$ have $\nu_Q(D_\ell) = \deg Q \cdot (n+1)$ by Lemma~\ref{L:nuP}.

We want to compute $\nu_P(D_{\ell})$. Write $f \in H^0(X, \mathcal O_X(kw+\ell))$ as
$$ f = \sum_{i=0}^k x_n^i \cdot f_i(x_0,\dots,x_{n-1}), \quad f_i \text{ homogenous of degree } (k-i)w+\ell.$$
If $\ell = 1$, then $f(P)$ = 0, and $f$ is not quasismooth at $P$ if and only if $f_k = 0$. As $f_k$ is a linear homogeneous polynomial in $n$ variables, this is a codimension $n$ condition, thus $\nu_P(D_1) = n$, compare Example~\ref{E:wps0}. With a similar computation, one obtains that $\nu_P(D_0) = 1$. However, if $\ell \geq 2$, then $f$ and all its partial derivatives automatically vanish at $P$. So the surjective map $\varphi_{Y_P,k}$ is the zero map, and consequently $\nu_P(D_\ell) = 0$.
\end{example}

\subsection{Medium degree points}

As we have seen in the previous example, we want to avoid low values of $\nu_P(D)$. For $m \geq 0$, define
$$\beta_m := \dim\, \{P \in Y \text{ closed} \mid \nu_P(D) = m \deg P \}.$$

\begin{lemma}[Medium degree points]\label{L:medium}
Fix an integer $r \geq 1$ and let $c$ be the constant from Lemma~\ref{L:surj2}. Let $Y_{r,sk}$ be the set of closed points $P$ of $Y$ with $r \leq \deg P \leq sk$, where
$$s := \frac{1}{\reg_E(D) \cdot \dim \pi^{-1}(Y)}.$$
 \begin{enumerate}[\normalfont (1)]
  \item If $\beta_m < m$ for all $m = 0, \dots, \dim Y$, then
\begin{align*}
  &\lim_{r \to \infty} \lim_{k \to \infty} \frac{\#\left\{ f \in H^0(X, \mathcal O_X(D+kE)) \middle| \begin{array}{c}
            Y \cap V(f) \text{ is not quasismooth}\\\text{at some } P \in Y_{r,sk}                                                                                         \end{array}
\right\}}{\#H^0(X, \mathcal O_X(D+kE))} = 0.
\end{align*}
 \item Otherwise
\begin{align*}
  \lim_{k \to \infty} \frac{\#\left \{ f \in H^0(X, \mathcal O_X(D+kE)) \middle| \begin{array}{c}
            Y \cap V(f) \text{ is not quasismooth}\\\text{at some } P \in Y_{r,sk}                                                                                         \end{array}
            \right \}}{\#H^0(X, \mathcal O_X(D+kE))} = 1.
\end{align*}
 \end{enumerate}
\end{lemma}
\begin{proof} \
\begin{enumerate}
 \item Let $k$ be a positive integer such that $k \geq \ell := \reg_E(D)$. Then we have the inequalities $k \cdot (1-\ell) \leq \ell \cdot (1-\ell)$ and thus
 $$k \leq k \ell - \ell^2 + \ell = \ell \cdot (k-\ell+1).$$
Hence, for $P \in Y_{r,sk}$,
$$ \deg P \leq \frac{k}{\ell \cdot \dim \pi^{-1}(Y)} \leq \frac{k-\ell+1}{\dim \pi^{-1}(Y)},$$
so the map $\varphi_{Y_P,k}$ is surjective by Corollary \ref{C:surj}. Following the proof of Lemma~\ref{L:low}, one finds that
$$\frac{\#\left\{ f \in H^0(X, \mathcal O_X(D+kE)) \middle|
\begin{array}{c}
 Y \cap V(f) \text{ is not}\\\text{quasismooth at }P                                                                                     \end{array} 
 \right \} }{\#H^0(X, \mathcal O_X(D+kE))}
 = q^{-\nu_P(D)}.$$
Hence we get the estimate
\begin{align*}
 &\frac{\#\left\{ f \in H^0(X, \mathcal O_X(D+kE)) \middle|
\begin{array}{c}
 Y \cap V(f) \text{ is not quasismooth}\\\text{at some } P \in Y_{r,sk}                                                                                       \end{array} 
 \right \} }{\#H^0(X, \mathcal O_X(D+kE))}\\
&\leq \sum_{e=r}^{sk} \sum_{P \in Y :\, \deg P = e} q^{-\nu_P(D)} \\
&\leq\sum_{e=r}^{sk} \sum_{m = 0}^{\dim \pi^{-1}(Y) } \sum_{P \in Y :\, \deg P = e, \nu_P(D) = em} q^{-em}, \quad k \geq \ell.
\end{align*}
Using the Lang-Weil bound \cite{LW}*{Theorem~1}, we can find a constant $L > 0$ such that 
$$\#\{P \in Y \mid \deg P = e, \nu_P(D) =  em\} \leq Lq^{e\beta_m}.$$
Hence
\begin{align*}
&\frac{\#\left\{ f \in H^0(X, \mathcal O_X(D+kE)) \middle|
\begin{array}{c}
 Y \cap V(f) \text{ is not quasismooth}\\\text{at some } P \in Y_{r,sk}                                                                                       \end{array} 
 \right \} }{\#H^0(X, \mathcal O_X(D+kE))}\\
&\leq \sum_{e=r}^{sk} \sum_{m = 0}^{\dim \pi^{-1}(Y) } Lq^{-e(m-\beta_m)}
\leq \sum_{m = 0}^{\dim \pi^{-1}(Y) } \sum_{e\geq 0} Lq^{-(e+r)(m-\beta_m)}\\
&=\sum_{m = 0}^{\dim \pi^{-1}(Y) } Lq^{-r(m-\beta_m)} \frac{1}{1-q^{\beta_m-m}}, \quad k \geq \ell.
\end{align*}
If $\beta_m < m$, this becomes arbitrarily small as $r \to \infty$.
\item Otherwise, choose an integer $m \in \{0, \dots, \dim Y\}$ and a subscheme $Z \subseteq Y$, $\dim Z \geq m$, such that for every closed point $P \in Z$ holds $\nu_P(D) = m \deg P$. For any integer $t \geq 0$, denote by $Z_{r,t}$ the finite set of closed points of $Z$ whose degree lies between $r$ and $t$. Further define for integers $k, t \geq 0$ the rational number
$$\quad a_{k,t} := \frac{\#\left\{ f \in H^0(X, \mathcal O_X(D+kE)) \middle|
\begin{array}{c}
 Y \cap V(f) \text{ is not quasismooth}\\\text{at some } P \in Z_{r,t}                                                                                       \end{array} 
 \right \} }{\#H^0(X, \mathcal O_X(D+kE))}.$$
By the techniques of Lemma~\ref{L:low},
\begin{align*}\lim_{k \to \infty} a_{k,t} &= 1 - \prod_{P \in Z_{r,t} \text{ closed}} \left(1 - q^{-\nu_P(D)} \right) \\
&= 1 - \prod_{P \in Z_{r,t} \text{ closed}} \left(1 - q^{-m \deg P} \right) \\
&= 1 - \prod_{P \in Z_{<r}}  \left(1 - q^{-m \deg P} \right)^{-1} \cdot \prod_{P \in Z_{\leq t}}  \left(1 - q^{-m \deg P} \right).
\end{align*}
The latter product vanishes if $m = 0$. Otherwise, we can use the standard power series expansion for the Hasse-Weil zeta function to obtain
$$  \prod_{P \in Z_{\leq t}} \left(1 - q^{-m \deg P} \right) = \exp \left(- \sum_{e=1}^t \#Z(\mathbb F_{q^e}) \frac{q^{-me}}{e} \right).$$
The Lang-Weil estimate \cite{LW}*{Theorem~1} gives a constant $M > 0$ depending on $Z$ such that $\#Z(\mathbb F_{q^e}) \geq Mq^{e \dim Z}$. Since $\dim Z \geq m$, the sum inside the exponential diverges to $\infty$ and therefore
$$ \lim_{t \to \infty} \lim_{k \to \infty} a_{k,t} = 1.$$
Let $\varepsilon > 0$. Then there is a number $t_\varepsilon$ such that 
$$ 1 - \varepsilon \leq \lim_{k \to \infty} a_{k,t_\varepsilon}.$$
Using the obvious inequality $a_{k,t_\varepsilon} \leq a_{k,sk}$ whenever $k \geq \frac{t_\varepsilon}{s}$ shows
$$  1 - \varepsilon  \leq \lim_{k \to \infty} a_{k,t_\varepsilon} \leq \underset{k \to \infty}{\lim\inf}\, a_{k,sk} \leq 1,$$
which completes the proof. \qedhere
\end{enumerate}
\end{proof}

\begin{remark}\label{R:medium}
The condition $\beta_m < m$ is automatically satisfied if $Y$ is smooth. It is still true if $Y$ has only finitely many singularities, provided that no point $P$ has $\nu_P(D) = 0$. We have already seen in Corollary~\ref{C:low} that the latter condition is necessary for having quasismooth intersections at all.
\end{remark}

\begin{example}\label{E:wps2}
Besides Example~\ref{E:wps1}, another example where the second case of Lemma~\ref{L:medium} applies is given by the following: Consider the weighted projective space $X = Y = \mathbb P(1,2,3,6)$ with coordinates $x_0, x_1, x_2, x_3$. Pick divisors $D$ and $E$ such that $\mathcal O_X(D) \cong \mathcal O_X(1)$ and $\mathcal O_X(E) \cong \mathcal O_X(6)$.

One computes that $\nu_P(D) = 1$ for any point $P \in V(x_0, x_1)$, thus $\beta_1 \geq \dim V(x_0,x_1) = 1$. In contrast to Example~\ref{E:wps1}, there is no point $P \in \mathbb P(1,2,3,6)$ with $\nu_P(D) = 0$. However, the hypersurfaces of degree $6k+1$ which are not quasismooth at some point in $V(x_0,x_1)$ still form a set of density one by Lemma~\ref{L:medium}~(2).
\end{example}

\subsection{High degree points}

We need two preparatorial lemmas.

\begin{lemma}\label{L:enough} Let $\ell := \reg_E(D)$.
\begin{enumerate}[\normalfont (1)]
\item Suppose that $X$ is smooth at the closed point $P$. Then, for $k \geq \ell$,
$$ \frac{\#\{f \in H^0(X, \mathcal O_X(D+kE)) \mid f(P) = 0 \}}{\# H^0(X, \mathcal O_X(D+kE))} \leq q^{-\min(k-\ell, \deg P)}.$$
\item Let $V \subseteq X$, $\dim V \geq 1$, be a subscheme which intersects the singular locus of $X$ in finitely many points only. Then
$$ \frac{\#\{ f \in H^0(X, \mathcal O_X(D+kE)) \mid V \subseteq \{f = 0\} \}}{\# H^0(X, \mathcal O_X(D+kE))} \leq q^{\ell-k}.$$
\end{enumerate}
\end{lemma}
\begin{proof}
Let $Z$ be the closed subscheme corresponding to the maximal ideal at $P$. Since $X$ is smooth at $P$, we have $H^0(Z, \mathcal O_X(D)|_Z) \cong H^0(Z, \mathcal O_Z)$, and the $\mathbb F_q$-dimension of this vector space equals $\deg P$. Assuming w.l.o.g. that $f_0(P) \neq 0$, the proof of Lemma~\ref{L:surj2} shows that the dimension of the image of the evaluation map
$$S_{[D+\ell  E]} \otimes \mathbb F_q[f_1,\dots,f_s]_{\leq k-\ell} \to H^0(X, \mathcal O_X(D+kE)) \xrightarrow{\varphi_{Z,k}} H^0(Z,\mathcal O_Z)$$
is at least $\min(k-\ell, \deg P)$. This proves (1). For (2), pick a point $P \in V$ contained in the smooth locus of $X$ such that $\deg P \geq k - \ell$.
\end{proof}
Note that the condition on smoothness is essential: Examples~\ref{E:wps1} and \ref{E:wps2} indicate that the fractions in question can be equal to one in the non-smooth case.

We need one more technical result. Let $W$ be a Weil divisor on $X$ and let $f \in S_{[W]}$ be a homogeneous polynomial of degree $[W]$ with respect to the grading given by the class group $\Cl(X)$. Since $S_{[W]} \subseteq \mathbb F_q[x_1,\dots,x_d]$, the polynomial $f$ carries a degree $\deg_\mathrm{std}(f)$ with respect to the standard grading on the polynomial ring $\mathbb F_q[x_1,\dots,x_d]$. Define
$$ \delta(W) := \max \,\{ \deg_\mathrm{std}(f) \mid f \in S_{[W]} \}.$$
\begin{lemma}\label{L:degree}
The quantity $\delta(D+kE)$ grows linearly in $k$.
\end{lemma}
\begin{proof}
 By Lemma~\ref{L:multiplication}, the natural multiplication map
$$ S_{[D+\ell E]} \otimes S_{[E]}^{\otimes (k-\ell)} \to S_{[D+kE]} $$
is surjective for $k \geq \ell := \reg_E(D)$. Consequently,
$$\delta(D+kE) = \delta(D+\ell  E) + (k-\ell) \cdot \delta(E), \quad k \geq \ell.$$
In particular, $\delta(D+kE)$ grows linearly in $k$.
\end{proof}

\begin{lemma}[High degree points]\label{L:high}
Fix a rational number $s > 0$ and denote by $Y_{>sk}$ the set of closed points of $Y$ of degree $> sk$. Suppose that $Y$ meets the singular locus of $X$ only in finitely many points. Then
\begin{align*}
  \underset{k \to \infty}{\lim\sup}\, \frac{\#\left\{ f \in H^0(X, \mathcal O_X(D+kE)) \middle|
\begin{array}{c}
 Y \cap V(f) \text{ is not quasismooth}\\\text{at some } P \in Y_{>sk}                                                                                       \end{array} 
 \right \} }{\#H^0(X, \mathcal O_X(D+kE))} = 0.
\end{align*}
\end{lemma}

\begin{proof}
We divide the proof into six steps. The strategy is as follows: We give first a global proof for $X = Y$. We choose an open cover of $X$ such that on each open, a hypersurface fails to be quasismooth if $\dim X$ many derivations vanish. Then we draw sections of $H^0(X, \mathcal O_X(D+kE))$ uniformly at random and compute that the probability that the locus where all derivations vanish contains a point of high degree. Applying Poonen's trick of decoupling derivatives, we show that this probability becomes arbitrarily small as $k \to \infty$. The last step is to generalize the proof to arbitrary quasismooth subschemes $Y \subseteq X$ with finitely many singular points.

\begin{step}
Testing quasismoothness with $n := \dim X$ many derivations.
\end{step}

Let $f \in S = K[x_1,\dots,x_d]$ be homogeneous with respect to the $\Cl(X)$-grading. Then, by the definition of quasismoothness, $V(f)$ is not quasismooth at $P \in X$ if and only if $$f(P) = \frac{\partial f}{\partial x_1}(P) = \dots = \frac{\partial f}{\partial x_d}(P) = 0.$$
In fact, even more is true: Let $\sigma \in \Sigma$ be an $n$-dimensional cone in the simplicial fan $\Sigma$ associated to $X$. The homogeneous coordinate ring $S$ has a variable $x_i$ for each one-dimensional cone $\rho_i \in \Sigma$, where $i = 1, \dots, d$. Define $U_\sigma \subseteq X$ to be the open affine subvariety given by the homogeneous localization at $\prod_{\rho_i \not\subseteq \sigma} x_i$. Renumbering the variables, we can assume that $\prod_{\rho_i \not\subseteq \sigma} x_i = x_{n+1} \cdots x_d$. By \cite{CoxQsm}*{Lemma~3.6}, if $P \in U_\sigma$, then $V(f)$ is not quasismooth at $P \in X$ if and only if 
$$f(P) = \frac{\partial f}{\partial x_1}(P) = \dots = \frac{\partial f}{\partial x_n}(P) = 0.$$
$X$ can be covered with finitely many such sets $U_\sigma$, and quasismoothness may be tested with $\dim X$ many derivations on each $U_\sigma$. So we may w.l.o.g. restrict our search for non-quasismooth points of high degree to $U_\sigma = \{x_{n+1} \dots x_d \neq 0\} \subseteq X$.
\begin{step}
Drawing sections at random.
\end{step}

Let $D_i$ be the divisor corresponding to $V(x_i)$, so that $x_i$ is a global section of $ \mathcal O_X(D_i)$, $i = 1, \dots, n$. Set $D_0 := 0 \in \mathrm{Div}(X)$. For $i = 0, \dots, n$ and $b = 0, \dots, q-1$, pick a divisor $\widetilde{C_{i,b}}$ such that $q \cdot \widetilde{C_{i,b}} \leq D+bE - D_i$, where $q$ is the cardinality of the ground field $\mathbb F_q$. Now fix an integer $k \geq 1$ and write $k = \lfloor k/q \rfloor \cdot q + b$. Define $C_i := \widetilde{C_{i,b}} + \lfloor k/q \rfloor \cdot E$. There is a natural multiplication map
$$H^0(X, \mathcal O_X(C_i)) \to H^0(X, \mathcal O_X(D+kE-D_i)), \quad g \mapsto g^q.$$
In order to see this, choose $g \in H^0(X, \mathcal O_X(C_i))$. Then
$$\mathrm{div}(g^q) = q \cdot \mathrm{div}(g) \geq q \cdot (- C_i) \geq -q \left\lfloor \frac{k}{q} \right\rfloor E -(D+bE-D_i) = -(D+kE-D_i),$$
hence $g^q \in H^0(X, \mathcal O_X(D+kE-D_i))$.
Note that for all $g \in H^0(X, \mathcal O_X(C_i))$,
$$ \frac{\partial g^q}{\partial x_j} = 0, \quad i = 0, \dots, n, \quad j = 1, \dots, n. $$
Combine these maps to
\begin{align*}
  \psi:
  \begin{array}{c}
    H^0(X, \mathcal O_X(D+kE))\\ 
   \oplus \\
    \bigoplus_{i=1}^n H^0(X, \mathcal O_X(C_i))\\
    \oplus \\
    H^0(X, \mathcal O_X(C_0))
  \end{array}
  &\to   H^0(X, \mathcal O_X(D+kE)),\\
  (f_0, g_1, \dots, g_n, h) &\mapsto f_0 + \sum_{i=1}^n g_i^q \cdot x_i +  h^q. 
\end{align*}
This map is $\mathbb F_q$-linear and surjective, hence we can compute densities on the left-hand side.

\begin{step}Decoupling of derivatives.\end{step}

For $f = \psi(f_0,g_1,\dots,g_n,h)$, define the subsets
$$ W_i := \left\{\frac{\partial f}{\partial x_1} = \dots = \frac{\partial f}{\partial x_i} = 0\right\} \subseteq  X \cap \{x_{n+1} \cdots x_d \neq 0\},\quad i = 0, \dots, n.$$
Note that $W_0$ is $n$-dimensional and for $i \geq 0$, $W_i$ does not depend on $g_{i+1}, \dots, g_n$ and $h$: Indeed, we have that
\begin{align*}
\frac{\partial f}{\partial x_i} &= \frac{\partial f_0}{\partial x_i} + \sum_{j=1}^m \frac{\partial x_j}{\partial x_i} \cdot g_j^q + \sum_{j=1}^m \underset{=0}{\underbrace{\frac{\partial g_j^q}{\partial x_i}}} \cdot x_j + \underset{=0}{\underbrace{\frac{\partial h^q}{\partial x_i}}}\\
&=  \frac{\partial f_0}{\partial x_i} + g_i^q, \quad i = 1, \dots, n.
\end{align*}

\begin{step} For $0 \leq i \leq n-1$, conditioned on a choice of $f_0, g_1, \dots, g_i$ for which $\dim W_i \leq n-i$, the probability that $\dim W_{i+1} \leq n-i-1$ is $1-o(1)$ as $k \to \infty$.
\end{step}

There is nothing to show if $\dim W_i \leq n-i-1$. Otherwise, if $\dim W_i = n - i$, the number of $(n-i)$-dimensional $\mathbb F_q$-irreducible components of $W_i$ is bounded from above by the number of $(d-i)$-dimensional $\mathbb F_q$-irreducible components of $\pi^{-1}(W_i)$, where $\pi: \mathbb A^d \setminus B \to X$ is the quotient map. Applying B\'ezout's theorem for affine space, this quantity is bounded by $O(\delta^i)$, where $\delta = \deg_\mathrm{std}(f)$ is the degree of $f \in \mathbb F_q[x_1,\dots,x_d]$ with respect to the standard grading.

Let $V$ be such an $(n-i)$-dimensional component of $W$. Define
$$ G_V^\text{bad} := \left\{ g_{i+1} \in H^0(X, \mathcal O_X(C_{i+1})) \left|\, V \subseteq \left\{\frac{\partial\psi(f_0,g_1,\dots,g_{i+1},*)}{\partial x_{i+1}} = 0 \right\} \right. \right\}.$$
Suppose that $G_V^\text{bad} \neq \emptyset$. If $g,g' \in G_V$, then $g^q-(g')^q = (g-g')^q$
 vanishes identically on $V \subseteq W_i$. So $g-g'$ must vanish identically on $V$. Hence there is a bijection
 $$ G_V^\text{bad} \leftrightarrow \{ g \in H^0(X, \mathcal O_X(C_{i+1})) \mid  V \subseteq \{g = 0 \} \}.$$
 Recall that $C_{i+1} = \widetilde{C_{i+1,b}} + \lfloor k/q \rfloor \cdot E$, where $k = \lfloor k/q \rfloor \cdot q + b$. Using Lemma~\ref{L:enough},
$$ \frac{\#G_V^\text{bad}}{\# H^0(X, \mathcal O_X(C_{i+1}))} = O(q^{-\lfloor k/q \rfloor}).$$
Since there are at most $O(\delta^i)$ such components $V$, and this number grows like $O(k^i)$ by Lemma~\ref{L:degree}, the probability that $W_{i+1}$ has dimension greater than $n-i-1$ is
$$ O(k^i q^{-\lfloor k/q \rfloor}) = o(1) \quad \text{ as } k \to \infty.$$

\begin{step}
Conditioned on a choice of $f_0, g_1, \dots, g_n$ for which $W_n$ is finite, the probability that $W_n \cap \{f = 0\}$ contains a point of degree $>sk$ is $o(1)$ as $k \to \infty$.
\end{step} 
We can follow the lines of the previous step: There is nothing to show if $W_n$ is empty. Otherwise, the number of points in $W_n$ is $O(k^n)$ again by B\'ezout's theorem and Lemma~\ref{L:degree}. Pick $P \in W_n$ and let
$$H_P^\text{bad} := \{h \in H^0(X, \mathcal O_X(C_0)) \mid \psi(f_0,g_1,\dots,g_n,h)(P) = 0\}.$$
Another application of Lemma~\ref{L:enough} yields that for all large enough $k$, either
$$ \frac{\#H_P^\text{bad}}{\# H^0(X, \mathcal O_X(C_0))} = O(q^{-\lfloor k/q \rfloor})$$
or $P$ is a singular point of $X$. The latter possibility can be ruled out since $X$ contains only finitely many singular points by hypothesis and $\deg P > sk$. As a consequence, the probability that $W_n \cap \{f = 0\}$ contains a point of degree $>sk$ is
$$ O(k^n q^{-\lfloor k/q \rfloor }) = o(1) \quad \text{ as } k \to \infty.$$
Putting everything together, the probability that a hypersurface $V(f)$, determined by choosing $f \in H^0(X, \mathcal O_X(D+kE))$ at random via $\psi$, is not quasismooth at some point in $P \in \{x_{n+1}\cdots x_d \neq 0\}$ of degree $>sk$ is $o(1)$ as $k \to \infty$. This proves the lemma in the case $X = Y$.

\begin{step}
Proof for general $Y$.
\end{step}
Following the strategy of the proof of \cite{Poonen}*{Lemma~2.6}, we can restrict to an open affine subset $U$ of the smooth locus $X^\text{sm}$ of $X$. We can find coordinates $t_1, \dots, t_n \in \mathcal O_U(U)$ defining $Y \cap X^\text{sm}$ locally by $t_{m+1} = \dots = t_n = 0$, where $m = \dim Y$. Moreover, there are derivations $d_1, \dots, d_m: \mathcal O_U(U) \to \mathcal O_U(U)$ such that for $f \in \mathcal O_U(U)$ and $P \in Y \cap U$,
\begin{align*}
 Y \cap V(f) \text{ is not quasismooth at } P  &\Leftrightarrow Y \cap V(f) \text{ is not smooth at } P \\
&\Leftrightarrow f(P) = d_1(f) = \dots = d_m(f) = 0.
\end{align*}
For $i = 1, \dots, m$, the coordinate $t_i$ may be considered as element of $\mathbb F_q(X) \cong \mathbb F_q(U)$, and thus
 $t_i \in H^0(X, \mathcal O_X(-\mathrm{div}(t_i)))$. This allows us to draw sections as in Step 2, replacing $D_i$ by $- \mathrm{div}(t_i)$. Restricting elements of $H^0(X, \mathcal O_X(D+kE))$ to $U$, the rest of the proof can be carried out analogously to the case $X = Y$.
\end{proof}

\begin{proof}[Proof of Theorem \ref{T:main}]
If $X$ happens to be zero-dimensional, then we are done by Lemma~\ref{L:low}. Otherwise, as in \cite{Poonen}*{§2.4}, the theorem follows from Lemmas \ref{L:low}, \ref{L:medium} and \ref{L:high} as $r \to \infty$.
\end{proof}

\section{Applications}\label{S:applications}

\subsection{First examples}

We list some easily obtained consequences of Theorem~\ref{T:main}:
\begin{enumerate}[(1)]
  \item Let $d_1, d_2, e_1, e_2 \in \mathbb Z$, $e_1, e_2 > 0$. Then as $k \to \infty$, the probability that a hypersurface of bidegree $(d_1+ke_1,d_2+ke_2)$ in $\mathbb P^m \times \mathbb P^n$ is smooth equals
  $$ \zeta_{\mathbb P^m \times \mathbb P^n}(m+n+1)^{-1} = \prod_{i=0}^m \prod_{j=0}^n (1-q^{i+j-m-n-1}),$$
as computed in \cite{Semiample}*{Example~4.3}.
  \item Let $w, \ell \in \mathbb Z$, $w \geq 1$, $0 \leq \ell \leq w-1$. As $k \to \infty$, the probability that a hypersurface of degree $kw+\ell$ is quasismooth in the weighted projective space $X = \mathbb P(1,\dots,1,w)$ of dimension $n$ equals
\begin{align*}
&0 &\text{if } \ell \geq 2, \\ 
% &\frac{1-q^{-n}}{1-q^{-n-1}} \cdot \zeta_{X}(n+1)^{-1} = 
&(1-q^{-1})\cdots(1-q^{-n+1}) \cdot (1-q^{-n})^2 &\text{if } \ell = 1, \\
% &\frac{1-q^{-1}}{1-q^{-n-1}} \cdot \zeta_{X}(n+1)^{-1} = 
&(1-q^{-1})^2 \cdot (1-q^{-2}) \cdots (1-q^{-n}) &\text{if } \ell = 0.
\end{align*}
This follows from the computations in Examples \ref{E:wps0} and \ref{E:wps1}. Moreover, as seen in Example~\ref{E:wps1}, in the case $\ell \geq 2$, every hypersurface passes through $(0:\dots:0:1)$ and is not quasismooth at this point.
\end{enumerate}

\subsection{Taylor conditions}

As in \cite{Poonen}*{Theorem~1.2}, there is an extended version of Theorem~\ref{T:main}:

\begin{theorem}\label{T:extended}
Let $X$ be a projective normal simplicial toric variety over a finite field $\mathbb F_q$. Fix a Weil divisor $D$ and an ample Cartier divisor $E$ on $X$. Let $Y \subseteq X$ be a quasismooth subscheme such that $Y$ meets the singular locus of $X$ only in finitely many points. Let $Z \subseteq X$ be a zero-dimensional subscheme and fix a subset $T \subseteq H^0(Z, \mathcal O_X(D)|_Z)$. Then
\begin{align*}
\lim_{k \to \infty} &\frac{\#\left\{f \in H^0(X, \mathcal O_{X}(D+kE)) \middle|
\begin{array}{c}
 (Y \setminus (Y \cap Z)) \cap V(f) \text{ is}\\
 \text{quasismooth and }\varphi_{Z,k}(f) \in T 
\end{array}
\right \} }{\#H^0(X, \mathcal O_{X}(D+kE))}\\
&= \frac{\#T}{\#H^0(Z, \mathcal O_X(D)|_Z)} \cdot \prod_{P \in Y \setminus (Y \cap Z) \text{ closed}} \left(1-q^{-\nu_P(D)}\right),
\end{align*}
where $\varphi_{Z,k}$ is the map as defined in subsection \ref{SS:surj}.
\end{theorem}
\begin{proof}
Since the set of sections in question is a subset of
$$ \{f \in H^0(X, \mathcal O_X(D+kE)) \mid (Y \setminus (Y \cap Z))\cap V(f) \text { is quasismooth}\},$$
we can apply the Lemmas \ref{L:medium} and \ref{L:high}. It suffices thus to modify the statement on low degree points. Let $Z'$ be the union of $Z$ with the zero-dimensional subscheme $Z$ used in the proof of Lemma~\ref{L:low}. Then a section $f \in H^0(X, \mathcal O_X(D+kE))$ is quasismooth at all $P$ in $(Y\setminus (Y \cap Z))_{<r}$ and $\varphi_{Z,k} \in T$ if and only if $f$ lies in the preimage of
$$ T \times \prod_{P \in(Y\setminus (Y \cap Z))_{<r}} \left(H^0(Y_P, \mathcal O_X(D)|_{Y_P}) \setminus \{0\} \right)$$
under the composition
\begin{align*}\varphi_{Z',k}: &H^0(X, \mathcal O_X(D+kE)) \to H^0(Z', \mathcal O_X(D)|_{Z'}) \\
&\xrightarrow{\simeq} H^0(Z, \mathcal O_X(D)|_Z) \times \prod_{P \in (Y\setminus (Y \cap Z))_{<r}}H^0(Y_P, \mathcal O_X(D)|_{Y_P}).\end{align*}
In virtue of Lemma~\ref{L:surj1}, this map becomes surjective for all sufficiently large $k$. Hence we can derive the formula given in the theorem. 
\end{proof}

As an application, let $Z$ be the zero-dimensional subscheme of all $\mathbb F_q$-rational points of $X$. Assume that no closed point $P \in X$ has $\nu_P(D) = 0$. Then $T := H^0(Z, \mathcal O_X(D)|_Z)\setminus\{0\}$ is non-empty and
\begin{align*}
\lim_{k \to \infty} &\frac{\#\left\{f \in H^0(X, \mathcal O_{X}(D+kE)) \middle|
\begin{array}{c}
(X \setminus Z) \cap V(f) \text{ is quasismooth}\\
\text{and } V(f)(\mathbb F_q) = \emptyset 
\end{array}
\right\}}{\#H^0(X, \mathcal O_{X}(D+kE))}\\
&= \frac{\#T}{\#H^0(Z, \mathcal O_X(D)|_Z)} \cdot \prod_{P \in X \setminus Z \text{ closed}} \left(1-q^{-\nu_P(D)}\right)\\
&> 0.
\end{align*}
In particular, for $k \gg 0$ exist quasismooth sections of $D+kE$ without $\mathbb F_q$-rational points.

\subsection{Singularities of positive dimension}

\begin{corollary}\label{C:positive}
With the notation of Theorem~\ref{T:main}, denote by $\mathrm{NQS}(f)$ the locus where the intersection $Y \cap V(f)$ is not quasismooth. Then
\begin{align*}
\underset{k \to \infty}{\lim\sup}\, &\frac{\#\{f \in H^0(X, \mathcal O_{X}(D+kE)) \mid \dim \mathrm{NQS}(f)\geq 1 \}}{\#H^0(X, \mathcal O_{X}(D+kE))} = 0.
\end{align*}
\end{corollary}
\begin{proof}
This follows immediately from Lemma~\ref{L:high}, as such an $f$ has a non-quasismooth point in $Y \cap V(f)$ of arbitrarily large degree.
\end{proof}

\subsection{Allowing a finite number of singularities}

\begin{theorem}\label{T:finite}
In the situation of Theorem~\ref{T:main}, suppose further that for any closed point $P \in Y$ holds $\nu_P(D) > 0$. Choose an integer $s \geq 1$. Then
\begin{align*}
\lim_{k \to \infty} &\frac{\#\left\{f \in H^0(X, \mathcal O_{X}(D+kE)) \middle|
\begin{array}{c}
Y \cap V(f) \text{ is quasismooth}\\
\text{except for} < s \text{ points}  
\end{array}
\right\}}{\#H^0(X, \mathcal O_{X}(D+kE))}\\
&=  \prod_{P \in Y \text{ closed}} (1 - q^{-\nu_P(D)}) \cdot \sum_{J \subseteq Y, \#J < s} \prod_{P \in J}\frac{1}{q^{\nu_P(D)} - 1}.
\end{align*}
\end{theorem}
\begin{proof}
Again, we can apply the strategy for medium and high degree points without big changes. So we take a look at low degree points. Fix an integer $r \geq 1$ and let $Y_{<r}$ be the set of closed points of $U$ of degree less than $r$. Denote again by $Z$ the union of all $Y_P$ for $P \in Y_{<r}$.

Recall that for $f \in H^0(X, \mathcal O_X(D+kE))$, the intersection $Y \cap V(f)$ is quasismooth at all points in $Y_{<r}$ if and only if all entries $\varphi_{Z,k}(f)$ are non-zero, where $\varphi_{Z,k}$ is the composition
\begin{align*}
H^0(X, \mathcal O_X(D+kE)) \to H^0(Z, \mathcal O_X(D)|_{Z}) \cong \prod_{P \in Y_{<r}} H^0(Y_P, \mathcal O_X(D)|_{Y_P})
\end{align*}
as in the proof of Lemma~\ref{L:low}.

In particular, the intersection $Y \cap V(f)$ is quasismooth at all points in $Y_{<r}$ except for less than $s$ points if and only if less than $s$ entries of $\varphi_{Z,k}(f)$ are zero.

Fix an enumeration $Y_{<r} = \{P_1, \dots, P_t\}$. If $0 \leq i < s$, then the number of elements in $ \prod_{P \in Y_{<r}} H^0(Y_P, \mathcal O_X(D)|_{Y_P})$ where precisely $i$ entries are zero is given by
$$ \sum_{1 \leq j_1 < \dots < j_i \leq t} \prod_{\ell \in \{1,\dots,t\}\setminus \{j_1,\dots, j_i\}} \left(\# H^0(Y_{P_\ell}, \mathcal O_X(D)|_{Y_{P_\ell}}) - 1\right).$$

Hence $Y \cap V(f)$ is quasismooth at all points $Y_{<r}$ except for less than $s$ points if and only if $f$ lies in the preimage of 
$$\sum_{i=0}^{s-1} \sum_{1 \leq j_1 < \dots < j_i \leq t} \prod_{\ell \in \{1,\dots,t\}\setminus \{j_1,\dots, j_i\}} \left(q^{\nu_{P_\ell}(D)} -1 \right)$$
elements under $\varphi_{k,Y}$.

By Lemma~\ref{L:surj1}, for any $r$ exists an integer $k_r$ such that $\varphi_{Z,k}$ is surjective for $k \geq k_r$. Thus for large enough $k$, the fibers of $\varphi_{Z,k}$ have the same cardinality.

Consequently,
\begin{align*}
  &\frac{\#\left\{ f \in H^0(X, \mathcal O_X(D+kE)) \middle|
  \begin{array}{c}
   Y \cap V(f) \text{ is quasismooth at all points}\\
   \text{in } Y_{<r} \text{ with } <s \text{ exceptions}
  \end{array}
  \right\}}{\#H^0(X, \mathcal O_X(D+kE))} \\
   &\quad= \frac{\sum_{i=0}^{s-1} \sum_{1 \leq j_1 < \dots < j_i \leq t} \prod_{\ell \in \{1,\dots,t\}\setminus \{j_1,\dots, j_i\}} \left( q^{\nu_{P_\ell}(D)} - 1\right) }{\prod_{\ell=1}^t q^{\nu_{P_\ell}(D)}} \\
   &\quad= \sum_{i=0}^{s-1}  \sum_{1 \leq j_1 < \dots < j_i \leq t} \prod_{\ell \in \{1,\dots,t\}\setminus \{j_1,\dots, j_i\}} \left(1 - q^{-\nu_{P_\ell}(D)} \right) \prod_{\ell=1}^i q^{-\nu_{P_{j_\ell}}(D)}\\
   &\quad=  \prod_{\ell=1}^t \left(1 - q^{-\nu_{P_\ell}(D)} \right)  \cdot \sum_{i=0}^{s-1}  \sum_{1 \leq j_1 < \dots < j_i \leq t} \prod_{\ell = 1}^i \frac{q^{-\nu_{P_{j_\ell}}(D)}}{1-q^{-\nu_{P_{j_\ell}}(D)}} \\
   &\quad= \prod_{P \in Y_{<r}} (1 - q^{-\nu_P(D)}) \cdot \sum_{J \subseteq Y_{<r}, \#J < s} \prod_{P \in J}\frac{1}{q^{\nu_P(D)} - 1}.
\end{align*}
It remains to show that 
$$\sum_{J \subseteq Y_{<r}, \#J < s} \prod_{P \in J}\frac{1}{q^{\nu_P(D)} - 1}$$
converges as $r \to \infty$. To this end, note that this is an increasing sequence as $r$ grows. So it suffices to give an absolute upper bound. Since
\begin{align*}
\sum_{J \subseteq Y_{<r}, \#J < s}\prod_{P \in J} \frac{1}{q^{\nu_P(D)} - 1} 
  &=\sum_{i=0}^{s-1} \sum_{\{P_1,\dots,P_i\} \subseteq Y_{<r}} \frac{1}{q^{\nu_{P_1}(D)} - 1} \cdots \frac{1}{q^{\nu_{P_i}(D)} - 1}
  \\
  &\leq \sum_{i=0}^{s-1} \left(\sum_{P \in Y_{<r}} \frac{1}{q^{\nu_P(D)} - 1}\right)^i,
\end{align*}
it suffices to bound $\sum_{P \in Y_{< r}} (q^{\nu_P(D)}-1)^{-1}$. By Lemma~\ref{L:nuP}, we have for all $P \in Y$ that $\nu_P(D) \leq \deg P \cdot \dim \pi^{-1}(Y)$ and $\nu_P(D) \geq \deg P$. Analogously to the proof of Lemma~\ref{L:medium},
\begin{align*}
\sum_{P \in Y_{<r}} \frac{1}{q^{\nu_P(D)} - 1}
&\leq \sum_{e=1}^{r-1} \sum_{m = 1}^{\dim \pi^{-1}(Y)} \frac{\#\{P \in Y \mid \deg P = e,  \nu_P(D) = em\}}{q^{e m} - 1}\\
    &\leq \sum_{m = 1}^{\dim \pi^{-1}(Y)} \sum_{e=1}^{r-1} \frac{C \cdot q^{e (m-1)}}{q^{em} - 1},
\end{align*}
for some constant $C$ not depending on $r$. Consequently,
\begin{align*}
\sum_{P \in Y_{<r}} \frac{1}{q^{\nu_P(D)} - 1}\leq C \cdot \sum_{m = 1}^{\dim \pi^{-1}(Y)} \sum_{e=1}^\infty \frac{1}{q^e - q^{-e(m-1)}}.
\end{align*}
Since $\sum_{e=1}^\infty (q^e - q^{-e(m-1)})^{-1}$ exists for $m \geq 1$, the expression on the left-hand side is bounded from above. Thus the desired limit exists.
\end{proof}

\begin{example}
 For $X = Y = \mathbb P^2$, the density of plane curves with at most one singular point is given by
$$ \frac{1}{\zeta_{\mathbb P^2}(3)} \cdot \left(1 + \sum_{P \in \mathbb P^2 \text{ closed}} \frac{1}{q^{3 \deg P}-1} \right). $$
For $q = 5$, this quantity is about $0.96984$.
\end{example}

We investigate now the density of hypersurfaces of degree $k$ whose number of singularities is bounded in terms of a strictly increasing function $k$.

\begin{lemma}\label{L:zeta} $\,$
\begin{enumerate}[\normalfont (1)]
  \item Let $(a_n)_{n \in \mathbb N}$ be a sequence of positive real numbers. Then for any $n \in \mathbb N$,
  $$ \sum_{J \subseteq \{1, \dots, n\}} \prod_{j \in J} a_j =  \prod_{j=1}^n (a_j + 1).$$
  \item Under the hypotheses of Theorem~\ref{T:finite},
$$ \lim_{s \to \infty} \sum_{J \subseteq Y, \#J < s} \prod_{P \in J} \frac{1}{q^{\nu_P(D)} - 1} =  \prod_{P \in Y \text{ closed}} \frac{1}{1-q^{-\nu_P(D)}}.$$
\end{enumerate}
\end{lemma}
\begin{proof}
Part (1) is easy. For (2), part (1) implies for any integer $r \geq 1$ the identity
\begin{align*}
 \lim_{s \to \infty} \sum_{J \subseteq Y_{<r}, \#J < s} \prod_{P \in J} \frac{1}{q^{\nu_P(D)} - 1}
 &=  \sum_{J \subseteq Y_{<r}} \prod_{P \in J} \frac{1}{q^{\nu_P(D)} - 1} \\
 &=  \prod_{P \in Y_{<r}} \frac{1}{1-q^{-\nu_P(D)}}.
\end{align*}
Taking limits,
$$ \lim_{r \to \infty} \lim_{s \to \infty} \sum_{J \subseteq Y_{<r}, \#J < s} \prod_{P \in J} \frac{1}{q^{\nu_P(D)} - 1} = \lim_{r \to \infty}  \prod_{P \in Y_{<r}} \frac{1}{1-q^{-\nu_P(D)}}.$$
Since the double sequence
$$ \left( \sum_{J \subseteq Y_{<r}, \#J < s} \prod_{P \in J} \frac{1}{q^{\nu_P(D)} - 1} \right)_{r,s}$$
is increasing and bounded, the iterated limits may be interchanged.
\end{proof}

\begin{corollary}\label{C:diagonal}
Let $g: \mathbb Z_{\geq 0} \to  \mathbb Z_{\geq 0} $ be a strictly increasing function. Then, under the hypotheses of Theorem~\ref{T:finite},
\begin{align*}
\lim_{k \to \infty} &\frac{\#\left\{f \in H^0(X, \mathcal O_{X}(D+kE)) \middle|
\begin{array}{c}
Y \cap V(f) \text{ is quasismooth}\\
\text{except for} < g(k) \text{ points} 
\end{array}
\right\}}{\#H^0(X, \mathcal O_{X}(D+kE))} = 1.
\end{align*}
\end{corollary}
\begin{proof}
 For integers $k \geq 0$, $s \geq 1$ define
\begin{align*}
a_{k,s} := \frac{\#\left\{f \in H^0(X, \mathcal O_{X}(D+kE)) \middle|
\begin{array}{c}
Y \cap V(f) \text{ is quasismooth}\\
\text{except for} < s \text{ points} 
\end{array}
\right\}}{\#H^0(X, \mathcal O_{X}(D+kE))}.
\end{align*}
Due to Theorem~\ref{T:finite} and Lemma~\ref{L:zeta},
$$\lim_{s\to \infty} \lim_{k\to\infty} a_{k,s} = 1.$$
Using the same reasoning as in the proof of Lemma~\ref{L:medium}~(2), one finds that for any given $\varepsilon > 0$, 
$$ 1 - \varepsilon \leq \underset{k \to \infty}{\lim\inf}\, a_{k,g(k)} \leq 1,$$
which proves the claim.
\end{proof}

\subsection{Length of the singular scheme}

As a final application, we show an analogue of Corollary~\ref{C:diagonal} for lengths of singular schemes of hypersurfaces on smooth toric varieties. Let $f \in S = \mathbb F_q[x_1,\dots,x_d]$ be a homogeneous polynomial. We endow the singular locus $\Sigma(f)$ of $f$ with the scheme structure given by the vanishing of the ideal $\left\langle f, \frac{\partial f}{\partial x_1}, \dots, \frac{\partial f}{\partial x_d} \right\rangle$. 

Pick a closed point $P \in X$ with local ring $\mathcal O_{X,P}$ and maximal ideal $\mathfrak m_{X,P}$. Since $X$ is smooth, we have a natural restriction map $S \to \mathcal O_{X,P}$. Define 
$$\length_P(\Sigma(f)) := \dim_{\mathbb F_q} \left. \mathcal O_{X,P}\middle/ \left\langle f, \frac{\partial f}{\partial x_1}, \dots, \frac{\partial f}{\partial x_d} \right\rangle \right. .$$
Then
$$ \length(\Sigma(f)) = \sum_{P \in X \text{closed }} \length_P(\Sigma(f)).$$
Suppose that $V(f)$ has only isolated singularities. Since isolated singularities are finitely determined \cite{Boubakri}*{Theorem~3}, $\length_P(\Sigma(f))$ depends only on the Taylor expansion of $f$ up to some degree. More precisely, for each integer $a \geq 0$ exists an $e_0 \geq 0$ such that for all integers $e \geq e_0$, we find a set $B_{P,a,e} \subseteq \mathcal O_{X,P}/\mathfrak m_{X,P}^{e}$ with the property that $\length_P(\Sigma(f)) = a$ if and only if $f$ lies in the preimage of $B_{P,a,e}$ under the natural map $S \to \mathcal O_{X,P}/\mathfrak m_{X,P}^{e}$. Write
$$\mu_P(a) :=  \frac{\#B_{P,a,e}}{\#\mathcal O_{X,P}/ \mathfrak m_{X,P}^e}.$$
Note that this quotient does not depend on the choice of $e$ due to finite determinacy. For example, $$\mu_P(0) = \frac{\#((\mathcal O_{X,P}/ \mathfrak m_{X,P}^2 )\setminus \{0\})}{\#\mathcal O_{X,P}/ \mathfrak m_{X,P}^2} = 1 - q^{-\deg P(\dim X +1)}.$$
We can now derive a result similar to Theorem~\ref{T:finite}:
\begin{theorem}\label{T:scheme}
In the situation of Theorem~\ref{T:main}, suppose further that $X$ is smooth. Choose an integer $s \geq 1$ and let
$$ A_s := \left\{(a_P)_{P \in X \text{ closed}} \middle | 
 a_P \in \{0, 1, \dots, s\} \text{ for all } P \in X \text{ closed and }
   \sum_{P \in X \text{ closed}} a_P < s 
 \right\}.$$
 Then
\begin{align*}
\lim_{k \to \infty} &\frac{\#\{f \in H^0(X, \mathcal O_{X}(D+kE)) \mid \length(\Sigma(f)) < s \}}{\#H^0(X, \mathcal O_{X}(D+kE))}\\
&=  \frac{1}{\zeta_X(\dim X + 1)} \cdot \sum_{a \in A_s} \prod_{P \in X \text{ closed}}\frac{\mu_P(a_P)}{\mu_P(0)}.
\end{align*}
\end{theorem}
\begin{proof}
In view of Corollary~\ref{C:positive}, we can restrict to hypersurfaces with isolated singularities. It is sufficient to perform the low degree computation and show convergence, the strategy for medium and high degree points being the same as previously. Fix an $r \geq 1$ and let $X_{<r} = \{P_1, \dots, P_t\}$ be the set of closed points of $X$ of degree $<r$. Let $(a_1, \dots, a_t)$ be a sequence of non-negative integers satisfying $a_1 + \dots + a_t = s$. Fix an integer $e$ being large enough to test whether $\length_{P_i}(\Sigma(f)) = a_i$ for all $i \in \{1, \dots, t\}$. The ideals $\mathfrak m_{X,P_i}^e$, $i = 1, \dots, t$, define zero-dimensional subschemes of $X$, let $Z$ denote their union. Then the natural map
$$ H^0(X, \mathcal O_X(D+kE)) \to H^0(Z, \mathcal O_Z) \cong \prod_{i=1}^t \mathcal O_{X,P_i}/\mathfrak m_{X,P_i}^e $$
becomes surjective for large enough $k$ due to Lemma~\ref{L:surj1}. Hence, imitating the proof of Lemma~\ref{L:low},
\begin{align*}
&\frac{\#\{f \in H^0(X, \mathcal O_{X}(D+kE)) \mid \length_{P_i}(\Sigma(f)) = a_i,\, i = 1, \dots, t\}}{\#H^0(X, \mathcal O_{X}(D+kE))} \\
&= \prod_{i=1}^t \mu_{P_i}(a_i), \quad k \gg 0.
\end{align*}
Consequently,
\begin{align*}
&\frac{\#\{f \in H^0(X, \mathcal O_{X}(D+kE)) \mid \length(\Sigma(f)) < s\}}{\#H^0(X, \mathcal O_{X}(D+kE))} \\
&\quad = \sum_{(a_1, \dots, a_t):\, \sum_{i=1}^t a_i < s} \prod_{i=1}^t \mu_{P_i}(a_i)\\
&\quad = \prod_{i=1}^t \mu_{P_i}(0) \cdot \sum_{(a_1, \dots, a_t):\, \sum_{i=1}^t a_i < s} \prod_{i=1}^t \frac{\mu_{P_i}(a_i)}{\mu_{P_i}(0)}\\
&\quad = \prod_{P \in X_{<r}} (1 - q^{-\deg P(\dim X + 1)}) \cdot \sum_{(a_P)_{P \in X_{<r}}:\, \sum_P a_P < s} \prod_{P \in X_{<r}} \frac{\mu_{P}(a_P)}{\mu_{P}(0)}
\end{align*}
for $k \gg 0$. The convergence of this expression follows from Theorem~\ref{T:finite}, as hypersurfaces $f$ with $\length(\Sigma(f)) < s$ have less than $s$ singular points.
\end{proof}

\begin{example}
 For $X = \mathbb P^2$, one finds
$$ \mu_P(1) = q^{-3 \deg P} - q^{-4 \deg P}, \quad P \in \mathbb P^2 \text{ closed}.$$
The density of plane curves with at most one ordinary double point as a singularity is therefore given by
$$ \frac{1}{\zeta_{\mathbb P^2}(3)} \cdot \left(1 + \sum_{P \in \mathbb P^2 \text{ closed}} \frac{1}{q^{\deg P} + q^{2 \deg P} + q^{3 \deg P}} \right). $$
For $q = 5$, this quantity is about $0.93113$.
\end{example}

\begin{corollary}\label{C:diagonalscheme}
In the situation of Theorem~\ref{T:scheme}, let $g: \mathbb Z_{\geq 0} \to  \mathbb Z_{\geq 0} $ be a strictly increasing function. Then
\begin{align*}
\lim_{k \to \infty} \frac{\#\{f \in H^0(X, \mathcal O_{X}(D+kE)) \mid \length(\Sigma(f)) < g(k)\}}{\#H^0(X, \mathcal O_{X}(D+kE))} = 1.
\end{align*}
\end{corollary}

\begin{proof}
Applying a similar strategy as in the proofs of Lemma~\ref{L:zeta} and Corollary~\ref{C:diagonal}, it suffices to show that
$$\lim_{s \to \infty} \sum_{\sum_P a_P < s} \prod_{P \in X_{<r}} \frac{\mu_{P}(a_P)}{\mu_P(0)} = \prod_{P \in X_{<r}} \frac{1}{1 - q^{-\deg P(\dim X + 1)}}, \quad r \geq 1,$$
or equivalently,
$$ \lim_{s \to \infty} \sum_{(a_P)_{P \in X_{<r}}:\, \sum_P a_P < s} \prod_{P \in X_{<r}}\mu_{P}(a_P) = 1, \quad r \geq 1.$$
This follows easily from the fact that
$$ \sum_{a \geq 0} \mu_P(a) = 1$$
for all closed points $P \in X$, which is a consequence of Theorem~\ref{T:scheme} and Corollary~\ref{C:positive}.
\end{proof}

\begin{remark}
 Over the complex numbers, it is known that the singular scheme of a non-factorial nodal hypersurface of degree $k \geq 3$ in $\mathbb P^4$ has length at least $(k-1)^2$, see e.g. \cite{Cheltsov}*{Theorem~1.4}. If an analogous result holds in positive characteristic, Corollary~\ref{C:diagonalscheme} will show that the density of non-factorial nodal hypersurfaces in $\mathbb P^4$ over $\mathbb F_q$ is zero.
\end{remark}

\vspace{2\baselineskip}

\begin{minipage}{0.45\textwidth}
\footnotesize
\textit{Address:}\\
{\scshape
Institut f\"ur Mathematik\\
Humboldt-Universit\"at zu Berlin\\
Unter den Linden 6\\
10099 Berlin\\
Germany}
\\[\baselineskip]
\textit{E-mail:} \texttt{lindnern@math.hu-berlin.de}
\end{minipage}
\begin{minipage}{0.5\textwidth}
\footnotesize
\textit{Current address:}\\
{\scshape
Institut f\"ur Algebraische Geometrie\\
Leibniz Universit\"at Hannover\\
Welfengarten 1\\
30167 Hannover\\
Germany}\\[\baselineskip]
\end{minipage}

\end{document}